\documentclass[a4paper]{article}
\usepackage[margin=1.2in]{geometry}
\usepackage{amsthm}
\usepackage{amssymb}
\usepackage{amsfonts}
\usepackage{subfigure}
\usepackage{amsmath,mathrsfs,bm}
\usepackage{array}
\usepackage{color,soul}
\usepackage{overpic}
\usepackage{ctable}
\usepackage{mathtools}
\usepackage{multirow}
\usepackage{authblk}

\newtheorem{theorem}{Theorem}[section]
\newtheorem{lemma}[theorem]{Lemma}
\newtheorem{definition}[theorem]{Definition}
\newtheorem{remark}[theorem]{Remark}

\numberwithin{equation}{section}

\begin{document}
\title{Two 11-node nonconforming triangular prism elements for 3D elliptic problems\thanks{This project is supported by NNSFC (Nos.~61733002, 61572096, 61432003, 61720106005, 61502107)
and ``the Fundamental Research Funds for the Central Universities''. }}

\author[a]{Xinchen Zhou\thanks{Corresponding author: dasazxc@gmail.com}}
\author[b]{Zhaoliang Meng}
\author[c]{Xin Fan}
\author[b,c]{Zhongxuan Luo}

\affil[a]{\it \small Faculty of Electronic Information and Electrical Engineering, Dalian University of Technology, Dalian 116024, China}
\affil[b]{\it \small School of Mathematical Sciences, Dalian
University of Technology, Dalian, 116024, China}
\affil[c]{\it School of Software, Dalian
University of Technology, Dalian, 116620, China}

\maketitle
\begin{abstract}
This work introduces two 11-node triangular prism elements for 3D elliptic problems.
The degrees of freedom (DoFs) of both elements are at the vertices and face centroids of a prism cell.
The first element is $H^1$-nonconforming and works for second order problems,
which achieves a second order convergence rate in discrete $H^1$-norm.
The other is $H^2$-nonconforming and solves fourth order problems,
with a first order convergence rate in discrete $H^2$-norm.
Numerical examples verify our theoretical results.
\\[6pt]
\end{abstract}

\section{Introduction}

Finite element methods are a powerful tool for partial differential equations.
An attractive issue is to design finite elements for their corresponding model problems,
but for the 3D case, it is generally not an easy task.
Perhaps the simplest model problem is the second order elliptic equation.
The family of $H^1$-conforming $P_k$ elements on tetrahedrons has become a typical example in many standard tutorials,
e.g.~\cite{Boffi2013,Brenner2002,Ciarlet1978,Chen2005}.
Over bricks,
the serendipity family is among the most popular finite elements,
for which a systematic study was given by Arnold and Awanou \cite{Arnold2011}.
Recently, $H^1$-nonconforming finite elements draw rapidly increasing attention
because of their simple structures and low brand widths in stiffness matrices.
In 1973, Crouzeix and Raviart \cite{Crouzeix1973} designed the centroid-continuous linear element over a simplex.
On cubical meshes, the 3D Wilson element \cite{Wilson1973},
the nonconforming rotated $Q_1$ elements \cite{Rannacher1992,Shi2000} and the nonconforming $P_1$ element \cite{Park2003}
are successful constructions.
Owing to the inter-element weak continuity or the interior symmetry of their finite element spaces,
these elements pass the generalized patch test \cite{Irons1972,Stummel1979,Shi1984,Shi1987}
and achieve the lowest order convergence rate.
Things become more complicated for higher order 3D nonconforming elements.
For simplicial meshes,
Fortin \cite{Fortin1985} extended the nonconforming quadratic triangular element \cite{Fortin1983} to the 3D case.
Several 14-node brick elements were proposed by Smith and Kidger \cite{Smith1992},
which are continuous at the vertices and the face centroids in a given cubical mesh,
but only parts of them are second order convergent in discrete $H^1$-norm \cite{Meng2016}.
Meng et al.~\cite{Meng2014} introduced another 14-node brick element using the same degrees of freedom (DoFs),
while the consistency error analysis is completely different.
Higher order 3D nonconforming elements with integral-type DoF selections can also be found in \cite{Matthies2005,Matthies2007,Meng2017}.

The construction of suitable finite elements for 3D fourth order elliptic problems is also an appealing subject.
Compared with $H^1$-conforming elements for second order problems,
it seems much more difficult to design an $H^2$-conforming element.
To meet the $C^1$-smoothness requirement,
many DoFs have to be utilized, and the degrees of the polynomials in shape function spaces must be very high.
An alternative way is to adopt the spline-based technique,
where a single element should be subdivided into several sub-elements.
Indeed, the polynomial constructions in \cite{Zenicek1973,Chen2009} and the spline-based construction \cite{Alfeld1984}
are the 3D extensions of their 2D counterparts: the Argyris element \cite{Argyris1968},
the Bogner-Fox-Schmit element \cite{Bogner1965} and the Hsieh-Clough-Tocher element \cite{Clough1965}, respectively.
To overcome the inconvenience of the $C^1$-continuity,
many low order $H^2$-nonconforming elements have been introduced.
On tetrahedrons, Wang et al.~\cite{Wang2006} extended the Morley element \cite{Morley1968} to the 3D case,
transferring the vertex DoFs in 2D into the edge DoFs in 3D.
They also provided several modifications of the 2D Zienkiewicz element \cite{Lascaux1975}
in a standard or non-standard manner \cite{Wang2007b,Wang2007c}.
On cubical meshes,
the 3D Morley element, the 3D Adini element, and the nonconforming 3D Bogner-Fox-Schmit element
were presented in \cite{Wang2007},
with nodal type DoFs at vertices employed.
For higher order nonconforming elements,
one can consult the contributions due to Wang et al.~\cite{Wang2012} and Chen et al.~\cite{Chen2013}
using a bubble function technique,
and the contributions by Hu et al.~\cite{Hu2016,Hu2016b} based on a superconvergence-like analysis.

All the aforementioned elements are constructed over tetrahedral or cubical meshes,
but much fewer works over triangular prisms have been reported in literature.
Triangular prisms are also commonly used in the partition of a 3D solution region.
In particular, they are more convenient and flexible than tetrahedral and brick cells
if the solution region is columnar with a complex base.
Moreover, they are often used for connecting other shapes containing triangular or rectangular faces.
It is a natural way of designing conforming triangular prism elements by utilizing the property that
a triangular prism is indeed the product of a triangle and an interval.
The $H^1$-conforming family over triangular prisms is a typical example \cite{Chen2005}.
Other examples include the mixed elements designed in \cite{Chen2005,Hu2018} for various problems.
As far as nonconforming elements are concerned, the above technique seems invalid.
We only find a direct extension of the nonconforming rotated $Q_1$ element \cite{Rannacher1992} to prisms by Arbogast and Chen \cite{Chen2005,Arbogast1995}, with the lowest convergence order achieved for the second order problem.
For fourth order problems, Chen et al.~\cite{Chen2013b,Chen2014} proposed several elements following the idea of their previous work \cite{Chen2013}.
It seems attractive and not trivial to design new triangular prism elements with lower computational costs,
while the convergence orders are as high as possible.

In this work, we construct two 11-node nonconforming triangular prism elements.
For each element, the shape function space is spanned by all the polynomials in $P_2$ plus a cubic one.
Both elements adopt vertex values as their DoFs,
but the remaining DoFs are different.
In the first element they are values at all face centroids,
and in the second element they are replaced by normal derivatives at all face centroids.
Explicit basis representations are provided for both theoretical and practical purposes.
Thanks to the special property of the element shape,
we can obtain optimal convergence results applied to their model problems.
The first element is $H^1$-nonconforming and works for second order elliptic problems,
which achieves a second order convergence rate in discrete $H^1$-norm.
The other is $H^2$-nonconforming and solves fourth order problems,
with a first order convergence rate in discrete $H^2$-norm.
Numerical examples are provided to confirm our theoretical findings.
To the best of our knowledge,
these two elements are of the lowest brand widths in stiffness matrices or of the least DoFs
among all known triangular prism elements in literature which preserve the same convergence orders for their corresponding model problems.

The rest of this work is arranged as follows.
We introduce the first element in Section \ref{s: element 1} with its unisolvency and basic properties,
and then provide the convergence analysis for second order elliptic problems.
Next, the counterpart for the second element is given in Section \ref{s: element 2} applied to fourth order problems.
We end this work with numerical tests in Section \ref{s: numer}.

Standard notations in Sobolev spaces are used throughout this work.
For a domain $D\subset\mathbb{R}^3$,
$\bm{n}$ will be the unit outward normal vector on $\partial D$.
The notation $P_k(D)$ denotes the usual polynomial space over $D$ of degree no more than $k$.
The norms and semi-norms of order $m$ in the Sobolev spaces $H^m(D)$
are indicated by $\|\cdot\|_{m,D}$ and $|\cdot|_{m,D}$, respectively.
The space $H_0^m(D)$ is the closure in $H^m(D)$ of $C_0^{\infty}(D)$.
We also adopt the convention that $L^2(D):=H^0(D)$,
where the inner-product is denoted by $(\cdot,\cdot)_D$.
These notations of norms, semi-norms and inner-products also work for vector- and matrix-valued Sobolev spaces,
where the subscript $\Omega$ will be omitted if the domain $D=\Omega$.
Moreover, $C$ is a positive constant independent of the mesh size $h$ in our model problems,
and may be different in different places.

\section{A nonconforming triangular prism element for 3D second order elliptic problems}
\label{s: element 1}

\subsection{Construction of the element and basic properties}
\label{ss: element 1}
Let $K$ be a triangular prism whose vertices are $V_1,V_2,\ldots,V_6$.
The side faces of $K$ are parallel to the $x_3$-axis,
which read $F_1=V_2V_3V_6V_5$, $F_2=V_3V_1V_4V_6$, $F_3=V_1V_2V_5V_4$.
Note that for each $i=1,2,3$, the vertices $V_i$ and $V_{i+3}$ are not in $F_i$.
The faces $F_4=V_1V_2V_3$ and $F_5=V_4V_5V_6$ are parallel to the $x_1Ox_2$-plane,
yielding that the $x_3$-coordinate of $F_5$ is greater than that of $F_4$.
For each $i=1,2,\ldots,5$, $M_i$ denotes the centroid of $F_i$.
See Figure \ref{fig: prism} as an example.

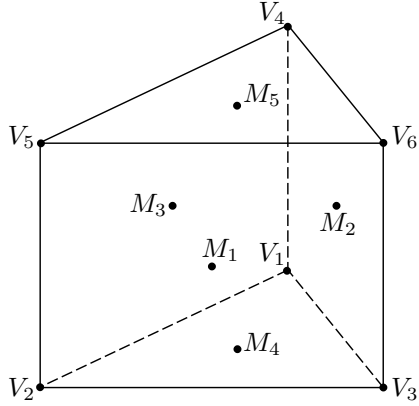
\begin{figure}[!htb]
\centering
\begin{overpic}[scale=0.27]{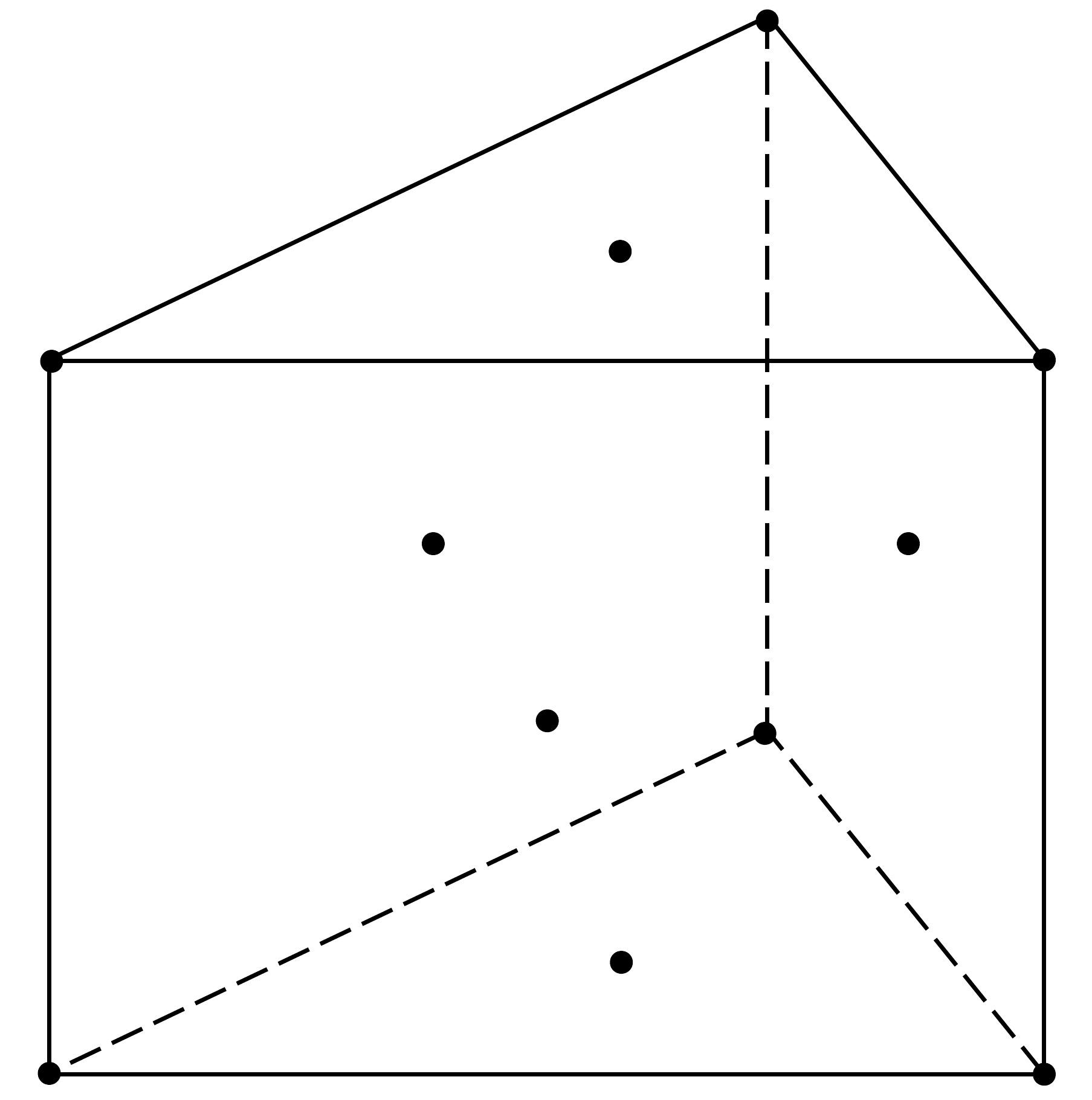}
\put(-4,0){$V_2$} \put(-4,67){$V_5$} \put(97,0){$V_3$} \put(97,67){$V_6$}
\put(62,36){$V_1$} \put(62,100){$V_4$} \put(47,37){$M_1$} \put(78,44){$M_2$}
\put(28,49){$M_3$} \put(58,12){$M_4$} \put(58,78){$M_5$}
\end{overpic}
\caption{Notations over a triangular prism $K$.\label{fig: prism}}
\end{figure}

For each $i=1,2,\ldots,5$,
the equation of the plane through $F_i$ is written as $\lambda_i(x_1,x_2,x_3)=0$,
which will be uniquely determined by $\lambda_i(V_i)=1$, $i=1,2,3$,
and $\lambda_4(M_1)=-\lambda_5(M_1)=1$.
We also set $\lambda_0(x_1,x_2,x_3)=0$ as the equation of the plane through $M_1,M_2$ and $M_3$
such that $\lambda_0(V_4)=-\lambda_0(V_1)=1$.
Indeed, $(\lambda_1,\lambda_2,\lambda_3)$ is the barycentric coordinate of the point $(x_1,x_2)$ in the face $F_4$
(or equivalently $F_5$),
and hence they are independent of $x_3$.
Similarly, $\lambda_0$, $\lambda_4$ and $\lambda_5$ are linear functions only in variable $x_3$.
The first 11-node triangular prism element is defined below.

\begin{definition}
\label{d: element 1}
The element $(K,P_K,\Sigma_K)$ is defined as follows:
\begin{itemize}
\setlength{\itemsep}{-\itemsep}
\item $K$ is the given triangular prism;
\item $P_K=P_2(K)\oplus\mathrm{span}\{\phi\}$ is the shape function space where
\[
\phi(x_1,x_2,x_3)=\frac{5}{12}\lambda_0\lambda_4\lambda_5+\lambda_0(\lambda_1\lambda_2
+\lambda_2\lambda_3+\lambda_3\lambda_1);
\]
\item $\Sigma_K=\{\sigma_i,~i=1,2,\ldots,11\}$ is the DoF set where
\[
\begin{aligned}
\sigma_i(v)&=v(V_i),~i=1,2,\ldots,6,\\
\sigma_{i+6}(v)&=v(M_i),~i=1,2,\ldots,5,
~\forall v\in P_K.
\end{aligned}
\]
\end{itemize}
\end{definition}

\begin{theorem}
\label{th: unisol 1}
The element $(K,P_K,\Sigma_K)$ is well-defined.
\end{theorem}
\begin{proof}
It suffices to give a nodal basis representation of $P_K$ with respect to $\Sigma_K$.
To this end, we first define
\[
\begin{aligned}
\phi_{i+6}&=\frac{8}{3}\lambda_j\lambda_k-\frac{4}{3}\lambda_i(1-\lambda_i)-\frac{1}{3}\lambda_4\lambda_5,
~i,j,k=1,2,3,~i\neq j\neq k\neq i,\\
\phi_{10}&=\frac{3}{2}(\lambda_1\lambda_2+\lambda_2\lambda_3+\lambda_3\lambda_1)+\frac{3}{8}\lambda_4\lambda_5-\frac{3}{2}\phi,\\
\phi_{11}&=\frac{3}{2}(\lambda_1\lambda_2+\lambda_2\lambda_3+\lambda_3\lambda_1)+\frac{3}{8}\lambda_4\lambda_5+\frac{3}{2}\phi.
\end{aligned}
\]
A careful calculation using the properties of all $\lambda_i$ leads to
\begin{equation}
\label{e: nodal 1 1}
\phi_{i+6}(V_j)=0,~\phi_{i+6}(M_k)=\delta_{i,k},~i,k=1,2,\ldots,5,~j=1,2,\ldots,6.
\end{equation}
Next, noting that for $i=1,2,3$ and $j=4,5$,
$\lambda_i\lambda_j$ vanishes at $V_k$, $M_i$ and $M_j$, $k\neq i$, $k=1,2,\ldots,6$,
we write
\[
\begin{aligned}
&\phi_i=-\frac{1}{2}\lambda_i\lambda_5-\frac{1}{4}(\phi_{j+6}+\phi_{k+6})-\frac{1}{3}\phi_{10},\\
&\phi_{i+3}=\frac{1}{2}\lambda_i\lambda_4-\frac{1}{4}(\phi_{j+6}+\phi_{k+6})-\frac{1}{3}\phi_{11},\\
&i,j,k=1,2,3,~i\neq j\neq k\neq i,
\end{aligned}
\]
then
\begin{equation}
\label{e: nodal 1 2}
\phi_i(V_j)=\delta_{i,j},~\phi_i(M_k)=0,~i,j=1,2,\ldots,6,~k=1,2\ldots,5.
\end{equation}
Hence, from (\ref{e: nodal 1 1}) and (\ref{e: nodal 1 2}) we find that $\phi_i\in P_K$ is the nodal basis function with respect to $\sigma_i$, $i=1,2,\ldots,11$,
which completes the proof.
\end{proof}

In what follows, we shall provide some crucial properties of this element.
Consider the following functions:
\[
\begin{aligned}
f_i(x_1,x_2,x_3)&=\lambda_i-(\lambda_1\lambda_2+\lambda_2\lambda_3+\lambda_3\lambda_1),~i=1,2,3,\\
f_4(x_1,x_2,x_3)&=3(\lambda_1\lambda_2+\lambda_2\lambda_3+\lambda_3\lambda_1).
\end{aligned}
\]
Since they are independent of $x_3$, one has
\begin{equation}
\label{e: face preserve 1}
P_1(F_j)=\mathrm{span}\{\lambda_i,~i=1,2,3\}\subset\mathrm{span}\{f_i,~i=1,2,3,4\},~j=4,5.
\end{equation}
Moreover, they yield
\[
\begin{aligned}
f_i(V_j)&=f_i(V_{j+3})=\delta_{i,j},~f_i(M_4)=f_i(M_5)=0,~i,j=1,2,3,\\
f_4(V_j)&=f_4(V_{j+3})=0,~j=1,2,3,~f_4(M_4)=f_4(M_5)=1.
\end{aligned}
\]
Therefore, for $F=F_4,F_5$ and $s>0$,
we can define the nodal interpolation operators $\mathcal{I}_F$ from
$H^{3/2+s}(K)$ to $\mathrm{span}\{f_i,~i=1,2,3,4\}$ by setting
\[
\begin{aligned}
(\mathcal{I}_{F_4}v)(V_i)&=v(V_i),~(\mathcal{I}_{F_5}v)(V_{i+3})=v(V_{i+3}),\\
(\mathcal{I}_{F_j}v)(M_j)&=v(M_j),~\forall v\in H^{3/2+s}(K),~i=1,2,3,~j=4,5.
\end{aligned}
\]

\begin{lemma}
\label{lemma: face equal 1}
For all $v\in P_K$, it holds that
\begin{equation}
\label{e: face equal 1}
v|_{F_4}-\mathcal{I}_{F_4}v=v|_{F_5}-\mathcal{I}_{F_5}v.
\end{equation}
\end{lemma}
\begin{proof}
First notice that the following functions $v_i=\lambda_i$, $v_{i+3}=\lambda_j\lambda_k$, $i,j,k=1,2,3$, $i\neq j\neq k\neq i$
are independent of $x_3$.
Therefore $v_i|_{F_4}=v_i|_{F_5}$, and so (\ref{e: face equal 1}) is true for $v_i$, $i=1,2,\ldots,6$.
Next, since $\lambda_0$ only depends on $x_3$,
for $v_{i+6}=\lambda_i\lambda_0$, $i=1,2,3$, $v_{10}=\lambda_0^2$, $v_{11}=\phi$,
we see $v_i|_{F_j}\in\mathrm{span}\{f_i,~i=1,2,3,4\}$, $i=7,8,\ldots,11$, $j=4,5$.
Hence $v_i|_{F_j}=\mathcal{I}_{F_j}v_i$, and thus (\ref{e: face equal 1}) also holds for $v_i$, $i=7,8,\ldots,11$.
Finally, the desired conclusion follows from the fact that $P_K=\mathrm{span}\{v_i,~i=1,2,\ldots,11\}$.
\end{proof}

The following lemmas reveal some useful relations on the faces of $K$.

\begin{lemma}
\label{lemma: face relation top 1}
For all $v\in P_K$, we have
\begin{equation}
\label{e: face relation top 1}
\begin{aligned}
\frac{1}{|F_4|}\int_{F_4}v\,\mathrm{d}\sigma&=\frac{1}{12}\sum_{i=1}^3v(V_i)+\frac{3}{4}v(M_4),\\
\frac{1}{|F_5|}\int_{F_5}v\,\mathrm{d}\sigma&=\frac{1}{12}\sum_{i=4}^6v(V_i)+\frac{3}{4}v(M_5).
\end{aligned}
\end{equation}
\end{lemma}
\begin{proof}
Owing to that $\phi|_{F_j}\in P_2(F_j)$, $j=4,5$,
it suffices to verify these equations for $v\in P_2(F_j)$,
which can be directly carried out by the 3-midpoint quadrature rule and the values of $\lambda_i$,
$i=1,2,3$ at associated points.
\end{proof}

\begin{lemma}
\label{lemma: face relation 1}
On the side faces of $K$, it holds that
\begin{equation}
\label{e: face relation 1}
\begin{aligned}
\frac{1}{|F_i|}\int_{F_i}v\,\mathrm{d}\sigma&=\frac{1}{12}\left(v(V_j)+v(V_k)+v(V_{j+3})+v(V_{k+3})\right)+\frac{2}{3}v(M_i),\\
\frac{1}{|F_i|}\int_{F_i}v\xi_i\,\mathrm{d}\sigma&=\frac{1}{12}\left(v(V_k)+v(V_{k+3})-v(V_j)-v(V_{j+3})\right),\\
\frac{1}{|F_i|}\int_{F_i}v\eta_i\,\mathrm{d}\sigma&=\frac{1}{12}\left(v(V_{j+3})+v(V_{k+3})-v(V_j)-v(V_k)\right),\\
&\forall v\in P_K,~i,j,k=1,2,3,~i\neq j\neq k\neq i,
\end{aligned}
\end{equation}
where $\xi_i$ and $\eta_i$ are the linear functions over $F_i$ such that
\[
\begin{aligned}
\xi_i(V_k)&=\xi_i(V_{k+3})=\eta_i(V_{j+3})=\eta_i(V_{k+3})=1,\\
\xi_i(V_j)&=\xi_i(V_{j+3})=\eta_i(V_j)=\eta_i(V_k)=-1.
\end{aligned}
\]
\end{lemma}
\begin{proof}
If $v\in P_2(K)$, then $v|_{F_i}\in P_2(F_i)$ and $v|_{F_i}\xi_i,v|_{F_i}\eta_i\in P_3(F_i)$, $i=1,2,3$.
Recall from \cite{Miller1960} that the 5-point diagonal Simpson quadrature rule on rectangles
(which is precisely the first equation in (\ref{e: face relation 1}))
is exact for bivariant polynomials of degree no more than three.
These equations are immediately obtained for $v\in P_2(K)$ by substituting the values of $v$,
$v\xi_i$ and $v\eta_i$ at the four vertices and the centroid of $F_i$.
Let us now turn to the case when $v=\phi$.
Since $\phi\in P_3(K)$, the quadrature rule above is still exact
and then the first equation in (\ref{e: face relation 1}) is valid.
Moreover, if $P$ and $P'$ in $F_i$ are symmetric points with respect to the line $\xi_i=0$,
then by the expression of $\phi$ we get $\phi(P)=\phi(P')$, and thus $(\phi\xi_i)(P)=-(\phi\xi_i)(P')$,
which ensures the validity of the second equation since both sides vanish.
Finally, the last equation can also be verified by using a higher order quadrature rule over $F_i$.
Indeed, both sides are again zero, which completes the proof.
\end{proof}

\subsection{Applied to second order elliptic problems}
\label{ss: poisson}
Let $\Omega\subset\mathbb{R}^3$ be a polyhedral domain which can be partitioned into triangular prism cells.
Consider the Poisson problem
\begin{equation}
\label{e: Poisson}
\begin{aligned}
-\Delta u&=f~~~&\mbox{in}~\Omega,\\
u&=0~&\mbox{on}~\partial\Omega,
\end{aligned}
\end{equation}
where $f\in H^{-1}(\Omega)$ is the given data.
The weak formulation of (\ref{e: Poisson}) is to find $u\in H_0^1(\Omega)$ such that
\begin{equation}
\label{e: weak 1}
(\nabla u,\nabla v)=(f,v),~~~\forall v\in H_0^1(\Omega).
\end{equation}

Let $\{\mathcal{T}_h\}$ be a family of quasi-uniform and shape-regular partitions
of $\Omega$ consisting of triangular prisms fulfilling the conditions in Subsection \ref{ss: element 1}.
The parameter $h$ is the maximum size of the cells in $\mathcal{T}_h$.
The sets of all interior vertices, boundary vertices,
faces, interior faces and boundary faces are correspondingly denoted by
$\mathcal{V}_h^i$, $\mathcal{V}_h^b$, $\mathcal{F}_h$, $\mathcal{F}_h^i$ and $\mathcal{F}_h^b$.
For each $F\in\mathcal{F}_h$, $\bm{n}_F$ is a fixed unit vector perpendicular to $F$,
and $M_F$ is the centroid of $F$.
Moreover, for $F\in\mathcal{F}_h^i$, the jump of a piecewise smooth function $v$
across $F$ is defined as $[v]_F=v|_{K_1}-v|_{K_2}$,
where $K_1$ and $K_2$ are the cells sharing $F$ as a common face.
For $F\in\mathcal{F}_h^b$, we set $[v]_F=v|_{K}$ if $F$ is a face of $K$.

The global finite element space $V_h$ is defined via
\[
\begin{aligned}
V_h=\Big\{v\in& L^2(\Omega):~v|_K\in P_K,~\forall K\in\mathcal{T}_h,
~\mbox{$v$ is continuous at all $V$ and $M_F$}\\
&\mbox{for $V\in\mathcal{V}_h^i$ and $F\in\mathcal{F}_h^i$, $v(V)=v(M_F)=0$ if $V\in \mathcal{V}_h^b$
and $F\in\mathcal{F}_h^b$}\Big\}.
\end{aligned}
\]
Then the finite element approximation of (\ref{e: weak 1}) is: Find $u_h\in V_h$ such that
\begin{equation}
\label{e: discrete weak 1}
\sum_{K\in\mathcal{T}_h}(\nabla u_h,\nabla v_h)_K=(f,v_h),~\forall v_h\in V_h.
\end{equation}
Introduce the semi-norm
\[
|v_h|_{m,h}=\left(\sum_{K\in\mathcal{T}_h}|v_h|_{m,K}^2\right)^{1/2},~m=1,2,
\]
then $|\cdot|_{1,h}$ is a norm on $V_h$,
and therefore problem (\ref{e: discrete weak 1}) has a unique solution due to the Lax-Milgram lemma.

For each $K\in\mathcal{T}_h$,
(\ref{e: face preserve 1}) implies that $\mathcal{I}_{F_j}v=v|_{F_j}$ for all $v\in P_1(K)$, $j=4,5$.
Then by the Bramble-Hilbert lemma, the trace theorem and a scaling argument, one has
\begin{equation}
\label{e: interp face err 1}
\|v-\mathcal{I}_{F_j}v\|_{0,F_j}\leq Ch^{k+1/2}|v|_{k+1,K},~\forall v\in H^2(K),~j=4,5,~k=0,1.
\end{equation}
For $s>0$, we define the interpolation operator $\mathcal{I}_K:H^{3/2+s}(K)\rightarrow P_K$
according to Theorem \ref{th: unisol 1} such that $\sigma_i(\mathcal{I}_Kv)=\sigma_i(v)$, $i=1,2,\ldots,11$.
Since $\mathcal{I}_Kv=v$ for all $v\in P_2(K)$, the Bramble-Hilbert lemma again gives
\begin{equation}
\label{e: interp err 1}
|v-\mathcal{I}_Kv|_{j,K}\leq Ch^{3-j}|v|_{3,K},~\forall v\in H^3(K),~j=0,1,2,3.
\end{equation}
The global interpolation operator $\mathcal{I}_h:H_0^1(\Omega)\cap H^{3/2+s}(\Omega)\rightarrow V_h$
is naturally set as $\mathcal{I}_h|_K=\mathcal{I}_K$, $\forall K\in\mathcal{T}_h$.

For each $K$, let $\mathcal{P}_{k,K}$ be the $L^2$-projection operator from $L^2(K)$ to $P_k(K)$, $k=0,1$.
Similarly, $\mathcal{P}_{k,F_i}$ denotes the linear operator on $H^1(K)$ $L^2$-projecting
$H^1(K)|_{F_i}$ to $P_k(F_i)$, $i=1,2,\ldots,5$, $k=0,1$.
The following result asserts a second order convergence rate in $|\cdot|_{1,h}$.

\begin{theorem}
\label{th: converge 1}
Let $u\in H_0^1(\Omega)\cap H^3(\Omega)$ and $u_h\in V_h$ be the solutions of (\ref{e: weak 1})
and (\ref{e: discrete weak 1}), respectively.
Then
\begin{equation}
\label{e: H1 1}
|u-u_h|_{1,h}\leq Ch^2|u|_3.
\end{equation}
\end{theorem}
\begin{proof}
The convergence analysis begins with the Strang lemma
\begin{equation}
\label{e: Strang}
|u-u_h|_{1,h}\leq C\left(\inf_{v_h\in V_h}|u-v_h|_{1,h}
+\sup_{w_h\in V_h}\frac{|E_{1,h}(u,w_h)|}{|w_h|_{1,h}}\right)
\end{equation}
with
\[
E_{1,h}(u,w_h)=\sum_{K\in\mathcal{T}_h}(\nabla u,\nabla w_h)_K-(f,w_h).
\]
It follows from (\ref{e: interp err 1}) that
\[
\inf_{v_h\in V_h}|u-v_h|_{1,h}\leq|u-\mathcal{I}_hu|_{1,h}\leq Ch^2|u|_3,
\]
therefore we only need to estimate the consistency error $E_{1,h}(u,w_h)$.

For the convenience of our description,
we rewrite the faces $F_i\subset\partial K$ as $F_{i,K}$, $i=1,2,\ldots,5$.
Integrating by parts elementwise gives
\begin{equation}
\label{e: E1h}
\begin{aligned}
E_{1,h}(u,w_h)&=\sum_{K\in\mathcal{T}_h}\int_{\partial K}\frac{\partial u}{\partial\bm{n}}w_h\,\mathrm{d}\sigma\\
&=\sum_{K\in\mathcal{T}_h}\int_{F_{1,K}\cup F_{2,K}\cup F_{3,K}}\frac{\partial u}{\partial\bm{n}}w_h\,\mathrm{d}\sigma
+\sum_{K\in\mathcal{T}_h}\int_{F_{4,K}\cup F_{5,K}}\frac{\partial u}{\partial\bm{n}}w_h\,\mathrm{d}\sigma\\
&:=I_1+I_2.
\end{aligned}
\end{equation}
Notice by (\ref{e: face relation 1}) and the weak continuity of $w_h$ that
\[
\int_F[w_h]_F\xi\,\mathrm{d}\sigma=0,~
\forall \xi\in P_1(F),~\mbox{if $F=F_{i,K}$ for some $K\in\mathcal{T}_h$, $i=1,2,3$}.
\]
Hence, we derive
\begin{equation}
\label{e: I1}
\begin{aligned}
|I_1|&=\left|\sum_{K\in\mathcal{T}_h}\sum_{i=1}^3\int_{F_{i,K}}
\left(\frac{\partial u}{\partial\bm{n}}-\mathcal{P}_{1,F_{i,K}}\frac{\partial u}{\partial\bm{n}}\right)\left(w_h-\mathcal{P}_{0,F_{i,K}}w_h\right)\,\mathrm{d}\sigma\right|\\
&\leq\sum_{K\in\mathcal{T}_h}\sum_{i=1}^3\left\|\frac{\partial u}{\partial\bm{n}}-\mathcal{P}_{1,F_{i,K}}\frac{\partial u}{\partial\bm{n}}\right\|_{0,F_{i,K}}\left\|w_h-\mathcal{P}_{0,F_{i,K}}w_h\right\|_{0,F_{i,K}}\\
&\leq\sum_{K\in\mathcal{T}_h}\sum_{i=1}^3Ch^{3/2}|u|_{3,K}Ch^{1/2}|w_h|_{1,K}
\leq Ch^2|u|_3|w_h|_{1,h}.
\end{aligned}
\end{equation}
As far as $I_2$ is concerned,
noting that if $F=F_{5,K_1}=F_{4,K_2}$ for some $K_1$ and $K_2$,
by the weak continuity of $w_h$ across $F$
we find $\mathcal{I}_{F_{5,K_1}}(w_h|_{K_1})=\mathcal{I}_{F_{4,K_2}}(w_h|_{K_2})$.
As a result,
\[
|I_2|=\left|\sum_{K\in\mathcal{T}_h}\left(
\int_{F_{5,K}}\frac{\partial u}{\partial x_3}
(w_h-\mathcal{I}_{F_{5,K}}w_h)\,\mathrm{d}\sigma
-\int_{F_{4,K}}\frac{\partial u}{\partial x_3}
(w_h-\mathcal{I}_{F_{4,K}}w_h)\,\mathrm{d}\sigma\right)\right|.
\]
Invoking (\ref{e: face equal 1}), one can rewrite $|I_2|$ as
\begin{equation}
\label{e: I2 temp}
\begin{aligned}
|I_2|=&\Bigg|\sum_{K\in\mathcal{T}_h}\Bigg(
\int_{F_{5,K}}\left(\frac{\partial u}{\partial x_3}-\mathcal{P}_{1,K}\frac{\partial u}{\partial x_3}\right)
(w_h-\mathcal{I}_{F_{5,K}}w_h)\,\mathrm{d}\sigma\\
&~~~~~~~~~~~~~~~~-\int_{F_{4,K}}\left(\frac{\partial u}{\partial x_3}-\mathcal{P}_{1,K}\frac{\partial u}{\partial x_3}\right)
(w_h-\mathcal{I}_{F_{4,K}}w_h)\,\mathrm{d}\sigma
\Bigg)\\
&+\sum_{K\in\mathcal{T}_h}\int_{F_{5,K}}
\left(\left(\mathcal{P}_{1,K}\frac{\partial u}{\partial x_3}\right)\Big|_{F_{5,K}}
-\left(\mathcal{P}_{1,K}\frac{\partial u}{\partial x_3}\right)\Big|_{F_{4,K}}\right)
(w_h-\mathcal{I}_{F_{5,K}}w_h)\,\mathrm{d}\sigma\Bigg|.
\end{aligned}
\end{equation}
Since $\mathcal{I}_{F_{5,K}}(w_h|_K)\in\mathrm{span}\{f_j,~j=1,2,3,4\}\subset P_K$
and $w_h|_K-\mathcal{I}_{F_{5,K}}(w_h|_K)$ vanishes at the three vertices and the centroid of $F_{5,K}$,
it follows from (\ref{e: face relation top 1}) that
\[
\int_{F_{5,K}}w_h-\mathcal{I}_{F_{5,K}}w_h\,\mathrm{d}\sigma=0.
\]
On the other hand, for any function $v\in P_1(K)$,
$v|_{F_{5,K}}-v|_{F_{4,K}}$ is a constant.
Therefore the second summation in (\ref{e: I2 temp}) vanishes,
which gives by (\ref{e: interp face err 1}) that
\begin{equation}
\label{e: I2}
\begin{aligned}
|I_2|&\leq\sum_{K\in\mathcal{T}_h}\sum_{i=4}^5\left\|\frac{\partial u}{\partial x_3}-\mathcal{P}_{1,K}\frac{\partial u}{\partial x_3}\right\|_{0,F_{i,K}}\left\|w_h-\mathcal{I}_{F_{i,K}}w_h\right\|_{0,F_{i,K}}\\
&\leq\sum_{K\in\mathcal{T}_h}\sum_{i=4}^5Ch^{3/2}|u|_{3,K}Ch^{1/2}|w_h|_{1,K}
\leq Ch^2|u|_3|w_h|_{1,h}.
\end{aligned}
\end{equation}
Finally, substituting (\ref{e: I1}) and (\ref{e: I2}) into (\ref{e: E1h}) we get
\[
|E_{1,h}(u,w_h)|\leq Ch^2|u|_3|w_h|_{1,h},~\forall w_h\in V_h,
\]
which completes the proof.
\end{proof}

\begin{remark}
If $\Omega$ is convex and $f\in L^2(\Omega)$,
then by a standard Aubin–Nitsche duality argument,
the following $L^2$-error estimate is easily obtained:
\begin{equation}
\label{e: L2 1}
\|u-u_h\|_0\leq Ch^3|u|_3.
\end{equation}
\end{remark}

\begin{remark}
According to (\ref{e: face relation top 1}) and (\ref{e: face relation 1}),
the face centroids can be directly replaced by face integrals in the DoF set $\Sigma_K$ for each $K\in\mathcal{T}_h$.
As a consequence, this may provide the ability for approximating more complicated problems.
A typical example is the following mixed finite element for the Stokes problem.
In fact, if $B_h$ denotes the elementwise bubble function space,
then the velocity space can be selected as $(V_h\oplus B_h)^3$,
while the pressure is approximated by piecewise discontinuous $P_1$ element.
According to the stability argument and the convergence analysis in \cite{Boffi2013},
this element is stable with a second order convergence rate for both the velocity and the pressure.
\end{remark}

\begin{remark}
Prisms are commonly used to connect other 3D shapes in hybrid meshes.
Let $\mathcal{T}_h$ be a mixed mesh consisting of both triangular prisms and bricks,
then we can select our element over triangular prism cells,
and select the 14-node Meng-Sheen-Luo-Kim element \cite{Meng2014} over brick cells.
The error estimates in (\ref{e: H1 1}) and (\ref{e: L2 1}) are still valid.
\end{remark}

\section{A nonconforming triangular prism element for 3D fourth order elliptic problems}
\label{s: element 2}

\subsection{Construction of the element and basic properties}

Under the same notations in Subsection \ref{ss: element 1}, we provide the second 11-node element over a triangular prism.

\begin{definition}
\label{d: element 2}
The element $(K,Q_K,T_K)$ is defined as follows:
\begin{itemize}
\setlength{\itemsep}{-\itemsep}
\item $K$ is the given triangular prism;
\item $Q_K=P_2(K)\oplus\mathrm{span}\{\lambda_0\lambda_4\lambda_5\}$ is the shape function space;
\item $T_K=\{\tau_i,~i=1,2,\ldots,11\}$ is the DoF set where
\[
\begin{aligned}
\tau_i(v)&=v(V_i),~i=1,2,\ldots,6,\\
\tau_{i+6}(v)&=\frac{\partial v}{\partial\bm{n}}(M_i),~i=1,2,\ldots,5,
~\forall v\in Q_K.
\end{aligned}
\]
\end{itemize}
\end{definition}

\begin{theorem}
\label{th: unisol 2}
The element $(K,Q_K,T_K)$ is well-defined.
\end{theorem}
\begin{proof}
We also show the unisolvency by giving a nodal basis representation of $Q_K$ with respect to $T_K$,
with the aid of the 2D Morley element over $F_4$ (or equivalently $F_5$).
Indeed, set
\[
\begin{aligned}
&q_i(x_1,x_2,x_3)=\frac{1}{\|\nabla\lambda_i\|}\lambda_i(\lambda_i-1),\\
&p_i(x_1,x_2,x_3)=\lambda_i+2\lambda_j\lambda_k-\nabla\lambda_j^T\nabla\lambda_k
\left(\frac{q_j}{\|\nabla\lambda_j\|}+\frac{q_k}{\|\nabla\lambda_k\|}\right),\\
&i,j,k=1,2,3,~i\neq j\neq k\neq i,
\end{aligned}
\]
then by Theorem 1 in \cite{Wang2006},
$\{p_i,q_i,~i=1,2,3\}$ is the nodal basis of the 2D Morley element over $F_4$ such that
\begin{equation}
\label{e: Morley}
\begin{aligned}
p_i(V_j)&=\frac{\partial q_i}{\partial\bm{n}_j}(\widetilde{M}_j)
=\frac{\partial q_i}{\partial\bm{n}_j}(M_j)=\delta_{i,j},\\
\frac{\partial p_i}{\partial\bm{n}_j}(\widetilde{M}_j)
&=\frac{\partial p_i}{\partial\bm{n}_j}(M_j)
=q_i(V_j)=0,~i,j=1,2,3,
\end{aligned}
\end{equation}
where $\widetilde{M}_i$ is the midpoint of $V_jV_k$,
and $\bm{n}_i$ is the unit outward normal vector of the face $F_i$,
or equivalently, the unit outward normal vector of $V_jV_k$ in the face $F_4$,
$i\neq j\neq k\neq i$, $i=1,2,3$.
Next, we define
\[
\begin{aligned}
\psi_i&=\frac{1}{2}\left(p_i-\lambda_0\left(\lambda_i-\frac{1}{6}\lambda_4\lambda_5\right)\right),\\
\psi_{i+3}&=\frac{1}{2}\left(p_i+\lambda_0\left(\lambda_i-\frac{1}{6}\lambda_4\lambda_5\right)\right),
~i=1,2,3.
\end{aligned}
\]
Then using (\ref{e: Morley}) and the values of $\lambda_i$ at the vertices we see
\begin{equation}
\label{e: nodal 2 1}
\psi_i(V_j)=\delta_{i,j},~i,j=1,2,\ldots,6.
\end{equation}
A closer observation shows that $\nabla\lambda_0=\nabla\lambda_4=\nabla\lambda_5$
and $\nabla\lambda_i^T\nabla\lambda_0=0$,
$i=1,2,3$, therefore by a direct computation and (\ref{e: Morley}),
\begin{equation}
\label{e: nodal 2 2}
\begin{aligned}
\frac{\partial\psi_i}{\partial\bm{n}}(M_j)
=&\frac{\nabla\lambda_j^T}{2\|\nabla\lambda_j\|}
\left(\lambda_0\nabla\lambda_i+
\left(\lambda_i-\frac{1}{6}(\lambda_4\lambda_5+\lambda_0\lambda_5+\lambda_0\lambda_4)
\right)\nabla\lambda_0
\right)\Bigg|_{(x_1,x_2,x_3)=M_j}=0,\\
\frac{\partial\psi_{i+3}}{\partial\bm{n}}(M_j)
=&-\frac{\nabla\lambda_j^T}{2\|\nabla\lambda_j\|}
\left(\lambda_0\nabla\lambda_i+
\left(\lambda_i-\frac{1}{6}(\lambda_4\lambda_5+\lambda_0\lambda_5+\lambda_0\lambda_4)
\right)\nabla\lambda_0
\right)\Bigg|_{(x_1,x_2,x_3)=M_j}=0,\\
&i=1,2,3,~j=1,2,3,4.
\end{aligned}
\end{equation}
A similar argument works for $j=5$ as well.
Moreover, we can also write
\[
\begin{aligned}
\psi_{i+6}&=q_i,~i=1,2,3,\\
\psi_{10}&=\frac{1}{4\|\nabla\lambda_0\|}\lambda_4\lambda_5(1-\lambda_0),\\
\psi_{11}&=\frac{1}{4\|\nabla\lambda_0\|}\lambda_4\lambda_5(1+\lambda_0).
\end{aligned}
\]
Notice that all $q_i$ are independent of $x_3$,
and $\psi_{10}$, $\psi_{11}$ are independent of $x_1,x_2$.
Thus,
\begin{equation}
\label{e: nodal 2 3}
\begin{aligned}
\tau_j(\psi_{i+6})&=0,~i=1,2,3,~j=10,11,\\
\tau_j(\psi_{10})&=\tau_j(\psi_{11})=0,~j=7,8,9.
\end{aligned}
\end{equation}
On the other hand, since
\[
\begin{aligned}
\frac{\partial\psi_{10}}{\partial x_3}(M_j)
&=\frac{\nabla\lambda_0^T}{\|\nabla\lambda_0\|}\cdot
\frac{1}{4\|\nabla\lambda_0\|}
\left((\lambda_4+\lambda_5)(1-\lambda_0)-\lambda_4\lambda_5\right)\nabla\lambda_0
\Big|_{(x_1,x_2,x_3)=M_j},\\
\frac{\partial\psi_{11}}{\partial x_3}(M_j)
&=\frac{\nabla\lambda_0^T}{\|\nabla\lambda_0\|}\cdot
\frac{1}{4\|\nabla\lambda_0\|}
\left((\lambda_4+\lambda_5)(1+\lambda_0)+\lambda_4\lambda_5\right)\nabla\lambda_0
\Big|_{(x_1,x_2,x_3)=M_j},~j=4,5,
\end{aligned}
\]
one must have
\begin{equation}
\label{e: nodal 2 4}
\tau_i(\psi_j)=\delta_{i,j},~i,j=10,11.
\end{equation}
Collecting (\ref{e: Morley})--(\ref{e: nodal 2 4}),
we see that $\psi_i\in Q_K$ is the nodal basis function with respect to $\tau_i$, $i=1,2,\ldots,11$,
which is the desired conclusion.
\end{proof}

For $F\subset\partial K$, define the nodal interpolation operators $\mathcal{J}_F$ from
$H^{3/2+s}(K)$ to $\mathrm{span}\{\lambda_i,~i=1,2,3\}=P_1(F_j)$, $j=4,5$ by setting
\[
(\mathcal{J}_{F_4}v)(V_i)=v(V_i),~(\mathcal{J}_{F_5}v)(V_{i+3})=v(V_{i+3}),~i=1,2,3,~j=4,5.
\]
The following lemmas are parallel to Lemmas \ref{lemma: face equal 1}--\ref{lemma: face relation 1}, respectively.

\begin{lemma}
For all $v\in Q_K$, it holds that
\begin{equation}
\label{e: face equal 2}
v|_{F_4}-\mathcal{J}_{F_4}v=v|_{F_5}-\mathcal{J}_{F_5}v.
\end{equation}
\end{lemma}
\begin{proof}
The proof is similar and simpler than that of Lemma \ref{lemma: face equal 1}, and thus omitted.
\end{proof}

\begin{lemma}
\label{lemma: face relation top 2}
For all $v\in Q_K$, we have
\begin{equation}
\label{e: face relation top 2}
\frac{1}{|F_j|}\int_{F_j}\frac{\partial v}{\partial\bm{n}}\,\mathrm{d}\sigma=\frac{\partial v}{\partial\bm{n}}(M_j),
~j=1,2,\ldots,5.
\end{equation}
\end{lemma}
\begin{proof}
A direct calculation shows
\[
\frac{\partial \lambda_0\lambda_4\lambda_5}{\partial \bm{n}}\Big|_{F_i}\in P_0(F_i),~i=1,2,\ldots,5.
\]
Therefore simply applying the centroid quadrature rule on triangles and rectangles will lead to the desired conclusion.
\end{proof}

\begin{lemma}
For each $i=1,2,3$, let $\bm{t}_i$ be a unit tangent vector on $F_i$ pointing from $V_j$ to $V_k$,
$1\leq j,k\leq 3$, $i\neq j\neq k\neq i$. Then
\begin{equation}
\label{e: face relation 2}
\begin{aligned}
\frac{1}{|V_jV_{j+3}|}\int_{F_i}\frac{\partial v}{\partial \bm{t}_i}\,\mathrm{d}\sigma
&=\frac{1}{2}(v(V_k)+v(V_{k+3})-v(V_j)-v(V_{j+3})),\\
\frac{1}{|V_jV_k|}\int_{F_i}\frac{\partial v}{\partial x_3}\,\mathrm{d}\sigma
&=\frac{1}{2}(v(V_{j+3})+v(V_{k+3})-v(V_j)-v(V_k)),\\
&\forall v\in Q_K,~i,j,k=1,2,3,~i\neq j\neq k\neq i.
\end{aligned}
\end{equation}
\end{lemma}
\begin{proof}
We first consider $v\in P_2(K)$.
To this end, note that the integrand in the first equation is a linear polynomial with respect to $x_3$.
Therefore, if we write $v(x_1,x_2,x_3)|_{F_i}=\widetilde{v}(s,x_3)$ with arc-length parameter $s$ on the edge $V_jV_k$,
and set $v(V_j)=\widetilde{v}(s_j,x_{3,j})$,
then the trapezoidal rule gives
\[
\begin{aligned}
\frac{1}{|V_jV_{j+3}|}\int_{F_i}\frac{\partial v}{\partial \bm{t}_i}\,\mathrm{d}\sigma
&=\int_{V_jV_k}\left(\frac{1}{|V_jV_{j+3}|}\int_{V_jV_{j+3}}\frac{\partial \widetilde{v}(s,x_3)}{\partial s}\,\mathrm{d}x_3\right)\mathrm{d}s\\
&=\int_{V_jV_k}\frac{1}{2}\left(\frac{\partial \widetilde{v}(s,x_{3,j})}{\partial s}
+\frac{\partial \widetilde{v}(s,x_{3,j+3})}{\partial s}\right)\,\mathrm{d}s\\
&=\frac{1}{2}(\widetilde{v}(s_k,x_{3,j})-\widetilde{v}(s_j,x_{3,j})+\widetilde{v}(s_k,x_{3,j+3})-\widetilde{v}(s_j,x_{3,j+3})),
\end{aligned}
\]
which is precisely the first equation for $v\in P_2(K)$.
The second equation can also be obtained in a similar manner.
Next we consider $v=\lambda_0\lambda_4\lambda_5$ independent of $x_1$ and $x_2$.
It is clear that $\widetilde{v}(s,x_3)$ is independent of $s$,
so both sides of the first equation vanish.
Moreover,
\[
\frac{1}{|V_jV_k|}\int_{F_i}\frac{\partial v}{\partial x_3}\,\mathrm{d}\sigma
=\int_{V_jV_{j+3}}\frac{\partial \widetilde{v}(s,x_3)}{x_3}\,\mathrm{d}x_3
=\widetilde{v}(s,x_{3,j+3})-\widetilde{v}(s,x_{3,j})=0,
\]
which verifies the second equation, and the proof is done.
\end{proof}

\subsection{Applied to fourth order elliptic problems}

Let us now consider the following fourth order elliptic problem:
\begin{equation}
\label{e: model problem}
\begin{aligned}
&\Delta^2u=f~~~&\mbox{in}~\Omega,\\
&u=\frac{\partial u}{\partial\bm{n}}=0~&\mbox{on}~\partial\Omega
\end{aligned}
\end{equation}
with $f\in L^2(\Omega)$.
The weak form of problem (\ref{e: model problem}) is to seek $u\in H_0^2(\Omega)$ such that
\begin{equation}
\label{e: weak 2}
(\nabla^2 u,\nabla^2 v)=(f,v),~\forall v\in H_0^2(\Omega),
\end{equation}
where $\nabla^2 v$ is the Hessian matrix of $v$.

Let $\mathcal{T}_h$ be given as in Subsection \ref{ss: poisson}.
We set
\[
\begin{aligned}
W_h=\Big\{v\in& L^2(\Omega):~v|_K\in Q_K,~\forall K\in\mathcal{T}_h,
~\mbox{$v$ is continuous at $V\in\mathcal{V}_h^i$ and}\\
&\mbox{vanishes at $V\in \mathcal{V}_h^b$, $\frac{\partial v}{\partial\bm{n}_F}$ is continuous at the centroid of $F$}\\
&\mbox{if $F\in\mathcal{F}_h^i$, and vanishes at the centroid of $F$ if $F\in\mathcal{F}_h^b$}\Big\}.
\end{aligned}
\]
Then the discrete form of (\ref{e: weak 2}) reads:
Find $u_h\in W_h$ such that
\begin{equation}
\label{e: discrete weak 2}
\sum_{K\in\mathcal{T}_h}(\nabla^2 u_h,\nabla^2 v_h)_K=(f,v_h),~\forall v_h\in W_h.
\end{equation}
Again $|\cdot|_{2,h}$ is a norm on $W_h$,
so problem (\ref{e: discrete weak 1}) has a unique solution.

For each $K\in\mathcal{T}_h$,
an analogous counterpart of (\ref{e: interp face err 1}) holds:
\begin{equation}
\label{e: interp face err 2}
\|v-\mathcal{J}_{F_j}v\|_{0,F_j}\leq Ch^{k+1/2}|v|_{k+1,K},~\forall v\in H^2(K),~j=4,5,~k=0,1.
\end{equation}
For $s>0$, the nodal interpolation operator $\mathcal{J}_K:H^{5/2+s}(K)\rightarrow Q_K$
due to Theorem \ref{th: unisol 2} is defined through $\tau_i(\mathcal{J}_Kv)=\tau_i(v)$, $i=1,2,\ldots,11$
with the interpolation error
\begin{equation}
\label{e: interp err 2}
|v-\mathcal{J}_Kv|_{j,K}\leq Ch^{3-j}|v|_{3,K},~\forall v\in H^3(K),~j=0,1,2,3.
\end{equation}
The global interpolation operator is then set as $\mathcal{J}_h|_K=\mathcal{J}_K$
from $H_0^2(\Omega)\cap H^{5/2+s}(\Omega)$ to $W_h$.

\begin{theorem}
\label{th: converge 2}
Let $u\in H_0^2(\Omega)\cap H^3(\Omega)$ and $u_h\in W_h$ be the solutions of (\ref{e: weak 2})
and (\ref{e: discrete weak 2}), respectively.
Then
\begin{equation}
\label{e: H2 2}
|u-u_h|_{2,h}\leq Ch(|u|_3+h\|f\|_0).
\end{equation}
\end{theorem}
\begin{proof}
Let us again consider the Strang lemma
\begin{equation}
\label{e: Strang 2}
|u-u_h|_{2,h}\leq C\left(\inf_{v_h\in W_h}|u-v_h|_{2,h}
+\sup_{w_h\in W_h}\frac{|E_{2,h}(u,w_h)|}{|w_h|_{2,h}}\right)
\end{equation}
with
\[
E_{2,h}(u,w_h)=\sum_{K\in\mathcal{T}_h}(\nabla^2 u,\nabla^2 w_h)_K-(f,w_h).
\]
Since (\ref{e: interp err 2}) gives
\[
\inf_{v_h\in W_h}|u-v_h|_{2,h}\leq|u-\mathcal{J}_hu|_{2,h}\leq Ch|u|_3,
\]
we turn to the consistency error $E_{2,h}(u,w_h)$.

For any $K\in\mathcal{T}_h$,
let $\Pi_K$ be the nodal interpolation operator from $H^{3/2+s}(K)$ to $\mathrm{span}\{\lambda_i,~i=1,2,3\}\times\mathrm{span}\{1,x_3\}$
such that $\Pi_Kv(V_i)=v(V_i)$, $i=1,2,\ldots,6$.
Then
\begin{equation}
\label{e: h1 interp err 2}
|v-\Pi_Kv|_{j,K}\leq Ch^{2-j}|v|_{2,K},~\forall v\in H^2(K),~j=0,1,2,
\end{equation}
and the global interpolation $\Pi_h|_K:=\Pi_K$, $\forall K\in\mathcal{T}_h$ will be $H^1$-conforming.
As before, we will adopt the convention that $F_{i,K}$ denotes $F_i\subset\partial K$ for $K\in\mathcal{T}_h$ in our
forthcoming description, $i=1,2,\ldots,5$.
By a density argument and integrating by parts, one has
\begin{equation}
\label{e: E2h}
\begin{aligned}
E_{2,h}(u,w_h)=&\sum_{K\in\mathcal{T}_h}(\nabla^2 u,\nabla^2 w_h)_K-(f,\Pi_hw_h)-(f,w_h-\Pi_hw_h)\\
=&-\sum_{K\in\mathcal{T}_h}(\nabla\Delta u,\nabla(w_h-\Pi_hw_h))_K-(f,w_h-\Pi_hw_h)\\
&+\sum_{K\in\mathcal{T}_h}\sum_{i,j=1}^3\int_{F_{1,K}\cup F_{2,K}\cup F_{3,K}}
\frac{\partial^2u}{\partial x_i\partial x_j}\frac{\partial w_h}{\partial x_i}n_j\,\mathrm{d}\sigma\\
&+\sum_{K\in\mathcal{T}_h}\sum_{i,j=1}^3\int_{F_{4,K}\cup F_{5,K}}
\frac{\partial^2u}{\partial x_i\partial x_j}\frac{\partial w_h}{\partial x_i}n_j\,\mathrm{d}\sigma\\
:=&J_1+J_2+J_3+J_4,
\end{aligned}
\end{equation}
where $\bm{n}=(n_1,n_2,n_3)^T$.
It follows from (\ref{e: h1 interp err 2}) that
\begin{equation}
\label{e: J12}
\begin{aligned}
|J_1|&\leq |u|_3|w_h-\Pi_hw_h|_{1,h}\leq Ch|u|_3|w_h|_{2,h},\\
|J_2|&\leq \|f\|_0\|w_h-\Pi_hw_h\|_0\leq Ch^2\|f\|_0|w_h|_{2,h},
\end{aligned}
\end{equation}
therefore it suffices to estimate $J_3$ and $J_4$.
Indeed, the weak continuity of $w_h$ and (\ref{e: face relation top 2}), (\ref{e: face relation 2}) imply that
\[
\int_F[\nabla w_h]_F\,\mathrm{d}\sigma=0,~\mbox{if $F=F_{i,K}$ for some $K\in\mathcal{T}_h$, $i=1,2,3$}.
\]
Hence,
\begin{equation}
\label{e: J3}
\begin{aligned}
|J_3|&=\left|\sum_{K\in\mathcal{T}_h}\sum_{k=1}^3\sum_{i,j=1}^3\int_{F_{k,K}}
\left(\frac{\partial^2u}{\partial x_i\partial x_j}-\mathcal{P}_{0,F_{k,K}}\frac{\partial^2u}{\partial x_i\partial x_j}\right)\left(\frac{\partial w_h}{\partial x_i}-\mathcal{P}_{0,F_{k,K}}\frac{\partial w_h}{\partial x_i}\right)n_j\,\mathrm{d}\sigma\right|\\
&\leq\sum_{K\in\mathcal{T}_h}\sum_{k=1}^3\sum_{i,j=1}^3\left\|\frac{\partial^2u}{\partial x_i\partial x_j}-\mathcal{P}_{0,F_{k,K}}\frac{\partial^2u}{\partial x_i\partial x_j}\right\|_{0,F_{k,K}}\left\|\frac{\partial w_h}{\partial x_i}-\mathcal{P}_{0,F_{k,K}}\frac{\partial w_h}{\partial x_i}\right\|_{0,F_{k,K}}\\
&\leq\sum_{K\in\mathcal{T}_h}Ch^{1/2}|u|_{3,K}Ch^{1/2}|w_h|_{2,K}
\leq Ch|u|_3|w_h|_{2,h}.
\end{aligned}
\end{equation}
On the other hand, on faces $F_{4,K}$ and $F_{5,K}$, $n_1=n_2=0$, therefore
\[
\begin{aligned}
J_4&=\sum_{K\in\mathcal{T}_h}\sum_{k=4}^5\int_{F_{k,K}}
\frac{\partial^2u}{\partial x_3^2}\frac{\partial w_h}{\partial x_3}n_3\,\mathrm{d}\sigma
+\sum_{K\in\mathcal{T}_h}\sum_{k=4}^5\sum_{i=1}^2\int_{F_{k,K}}
\frac{\partial^2u}{\partial x_i\partial x_3}\frac{\partial w_h}{\partial x_i}n_3\,\mathrm{d}\sigma\\
&:=J_{4,1}+J_{4,2}.
\end{aligned}
\]
A similar argument as in (\ref{e: J3}) using the weak continuity of $w_h$ and (\ref{e: face relation top 2}) derives
\begin{equation}
\label{e: J41}
|J_{4,1}|\leq Ch|u|_3|w_h|_{2,h}.
\end{equation}
As far as $J_{4,2}$ is concerned,
if $F=F_{5,K_1}=F_{4,K_2}$ for some $K_1$ and $K_2$,
the weak continuity of $w_h$ across $F$ gives
$\mathcal{J}_{F_{5,K_1}}(w_h|_{K_1})=\mathcal{J}_{F_{4,K_2}}(w_h|_{K_2})$,
which yields
\[
\begin{aligned}
|J_{4,2}|&=\Bigg|\sum_{K\in\mathcal{T}_h}\sum_{i=1}^2\Bigg(
\int_{F_{5,K}}\frac{\partial^2 u}{\partial x_i\partial x_3}
\left(\frac{\partial w_h}{\partial x_i}-\frac{\partial \mathcal{J}_{F_{5,K}}w_h}{\partial x_i}\right)\,\mathrm{d}\sigma\\
&~~~~~~~~~~~~~~~~~~-\int_{F_{4,K}}\frac{\partial^2 u}{\partial x_i\partial x_3}
\left(\frac{\partial w_h}{\partial x_i}-\frac{\partial \mathcal{J}_{F_{4,K}}w_h}{\partial x_i}\right)\,\mathrm{d}\sigma\Bigg)\Bigg|.
\end{aligned}
\]
Then using (\ref{e: face equal 2}), the inverse inequality and (\ref{e: interp face err 2}), we find
\begin{equation}
\label{e: J42}
\begin{aligned}
|J_{4,2}|&=\Bigg|\sum_{K\in\mathcal{T}_h}\sum_{i=1}^2\Bigg(
\int_{F_{5,K}}\left(\frac{\partial^2 u}{\partial x_i\partial x_3}-
\mathcal{P}_{0,K}\frac{\partial^2 u}{\partial x_i\partial x_3}\right)
\left(\frac{\partial w_h}{\partial x_i}-\frac{\partial \mathcal{J}_{F_{5,K}}w_h}{\partial x_i}\right)\,\mathrm{d}\sigma\\
&~~~~~~~~~~~~~~~~~~-\int_{F_{4,K}}\left(\frac{\partial^2 u}{\partial x_i\partial x_3}-
\mathcal{P}_{0,K}\frac{\partial^2 u}{\partial x_i\partial x_3}\right)
\left(\frac{\partial w_h}{\partial x_i}-\frac{\partial \mathcal{J}_{F_{4,K}}w_h}{\partial x_i}\right)\,\mathrm{d}\sigma\Bigg)\Bigg|\\
&\leq \sum_{K\in\mathcal{T}_h}\sum_{k=4}^5\sum_{i=1}^2
\left\|\frac{\partial^2 u}{\partial x_i\partial x_3}-\mathcal{P}_{0,K}\frac{\partial^2 u}{\partial x_i\partial x_3}\right\|_{0,F_{k,K}}
\left|w_h-\mathcal{J}_{F_{5,K}}w_h\right|_{1,F_{k,K}}\\
&\leq \sum_{K\in\mathcal{T}_h}\sum_{k=4}^5\sum_{i=1}^2
Ch^{1/2}\left|\frac{\partial^2 u}{\partial x_i\partial x_3}\right|_{1,K}Ch^{-1}
\left\|w_h-\mathcal{J}_{F_{5,K}}w_h\right\|_{0,F_{k,K}}\\
&\leq \sum_{K\in\mathcal{T}_h}\sum_{k=4}^5\sum_{i=1}^2
Ch^{1/2}|u|_{3,K}Ch^{-1}Ch^{3/2}|w_h|_{2,K}
\leq Ch|u|_3|w_h|_{2,h}.
\end{aligned}
\end{equation}
Substituting (\ref{e: J12})--(\ref{e: J42}) into (\ref{e: E2h}), one has
\[
|E_{2,h}(u,w_h)|\leq Ch(|u|_3+h\|f\|_0)|w_h|_{2,h},
\]
which completes the proof.
\end{proof}

\begin{remark}
We again consider a mixed mesh $\mathcal{T}_h$ consisting of both triangular prisms and bricks.
This element works for triangular prism cells,
while one can select the 3D rectangular Morley element \cite{Wang2007} designed by Wang, Shi and Xu over brick cells.
The error estimate (\ref{e: H2 2}) still holds.
\end{remark}

\section{Numerical examples}
\label{s: numer}

Numerical examples are provided in this section for both elements and their corresponding model problems.
Let the solution domain $\Omega$ be the unit cube $[0,1]^3$.
For each even number $n>0$,
a suitable triangular prism partition $\mathcal{T}_h$ associated with $n$ can be generated as follows.
In the first step, we construct an $n\times n$ uniform trapezoidal partition over the 2D region $[0,1]^2$,
and then divide each trapezoidal cell into two triangles,
so that a 2D triangulation $\mathcal{T}_h^{\mathrm{2D}}$ is obtained.
See Figure \ref{fig: subfig: 2D} as an example for $n=4$.
The second step is to uniformly divide the $[0,1]$ interval into $n$ pieces to obtain $\mathcal{T}_h^{\mathrm{1D}}$,
and then we can construct $\mathcal{T}_h$ by setting
\[
\mathcal{T}_h=\left\{K\subset\Omega:~K=K^{\mathrm{2D}}\times K^{\mathrm{1D}},
~K^{\mathrm{2D}}\in\mathcal{T}_h^{\mathrm{2D}},~K^{\mathrm{1D}}\in\mathcal{T}_h^{\mathrm{1D}}\right\}.
\]
Figure \ref{fig: subfig: mesh} illustrates an example of $\mathcal{T}_h$ when $n=4$.

\begin{figure}[!htb]
\centering
\subfigure[The 2D triangulation $\mathcal{T}_h^{\mathrm{2D}}$.] {
\label{fig: subfig: 2D}
\includegraphics[scale=0.40]{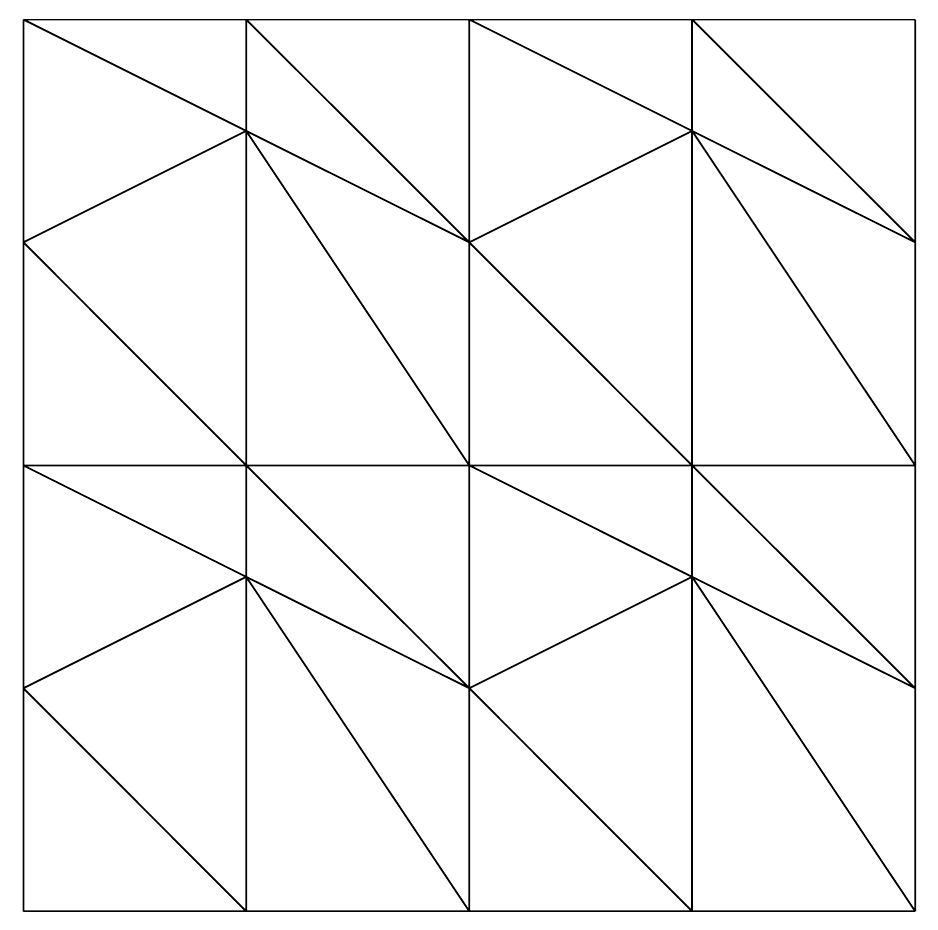}
}
\hspace{0.5cm}
\subfigure[The partition $\mathcal{T}_h$ over $\Omega$.] {
\label{fig: subfig: mesh}
\includegraphics[scale=0.22]{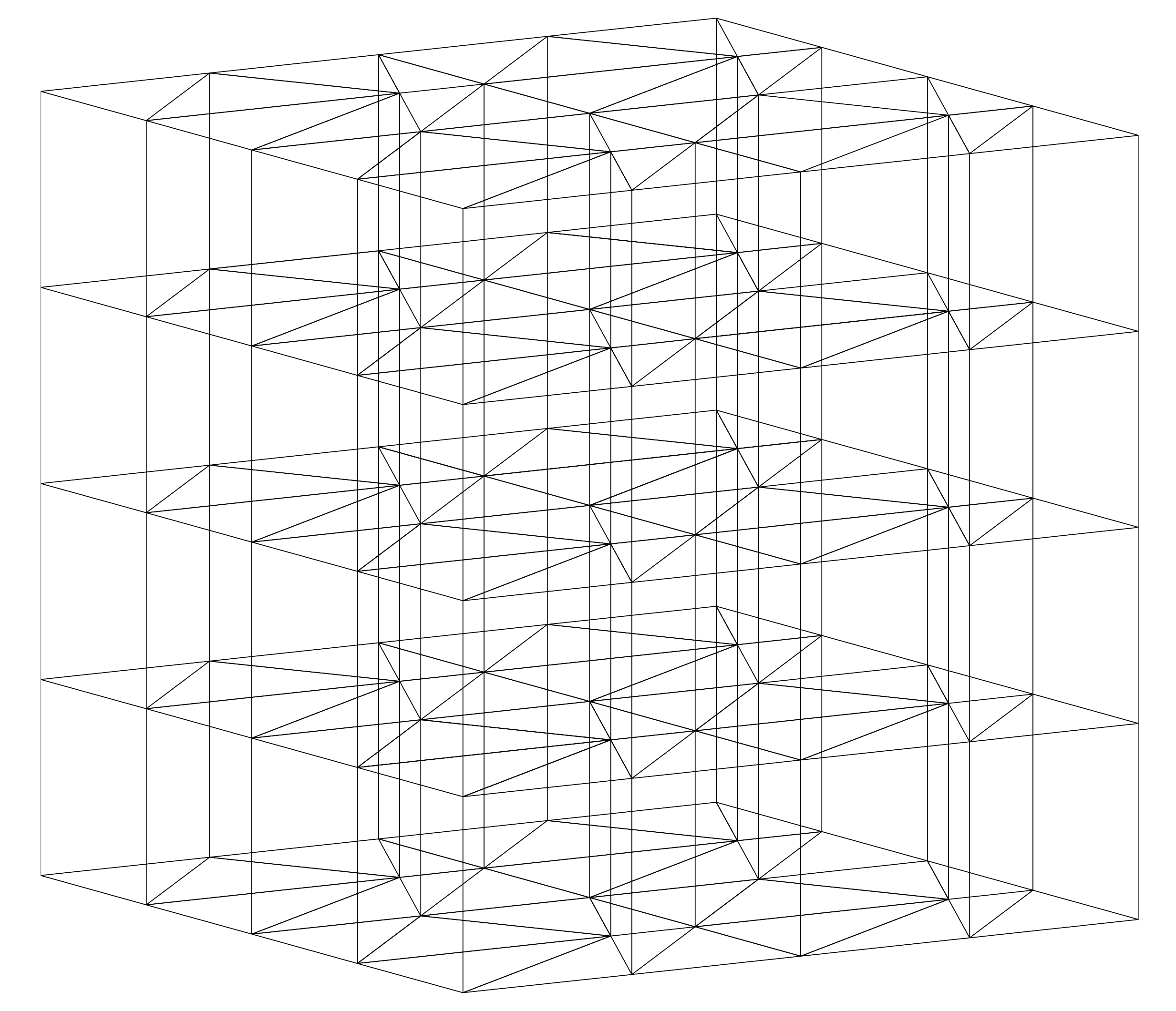}
}
\caption{Generation of the partition $\mathcal{T}_h$ when $n=4$.
\label{fig: partitions}}
\end{figure}

Let us first check the performance of the first element for 3D Poisson problems.
The exact solution of (\ref{e: Poisson}) is given by
\begin{equation}
\label{e: example 1}
u=\exp(x_1-2\pi x_2+3\pi x_3)\sin(2\pi x_2)\sin(3\pi x_3)(x_1^2-x_1^3).
\end{equation}
In Table \ref{t: example 1},
we list the errors in discrete $H^1$-norm $|u-u_h|_{1,h}$ with various $n$.
As predicted in Theorem \ref{th: converge 1},
a second order convergent behavior is observed.
The $L^2$ errors are also provided with a third order convergence rate,
which is consistent with (\ref{e: L2 1}).

\begin{table}[!htb]
\begin{center}
\begin{tabular}{p{1cm}<{\centering}p{2cm}<{\centering}p{1cm}<{\centering}p{2cm}<{\centering}p{1cm}<{\centering}}
\toprule
$n$ & $|u-u_h|_{1,h}$ & order & $\|u-u_h\|_0$ & order \\
\midrule
$4$      &4.035E2     &       &1.312E1    &          \\
$8$      &2.272E2     &0.83   &4.394E0    &1.58      \\
$16$     &7.840E1     &1.54   &7.557E-1   &2.54      \\
$32$     &2.142E1     &1.87   &9.250E-2   &3.03      \\
$64$     &5.480E0     &1.97   &1.107E-2   &3.06      \\
\bottomrule
\end{tabular}
\caption{The discrete $H^1$ and $L^2$ errors produced by $V_h$ applied to the Poisson problem determined through (\ref{e: example 1}) for various $n$.
 \label{t: example 1}}
\end{center}
\end{table}

In what follows, we turn to the performance of the second element applied to 3D fourth order elliptic problems.
The exact solution of (\ref{e: model problem}) is determined by
\begin{equation}
\label{e: example 2}
u=(1+\cos((2x_1-1)\pi))(1+\cos((2x_2-1)\pi)) (1+\cos((2x_3-1)\pi)).
\end{equation}
Table \ref{t: example 2} shows the errors $|u-u_h|_{2,h}$.
Again a first order convergence rate is achieved as predicted in Theorem \ref{th: converge 2}.
We also notice that the discrete $H^1$ and $L^2$ errors are of second order convergence.

\begin{table}[!htb]
\begin{center}
\begin{tabular}{p{1cm}<{\centering}p{2cm}<{\centering}p{1cm}<{\centering}p{2cm}<{\centering}p{1cm}<{\centering}
p{2cm}<{\centering}p{1.5cm}<{\centering}}
\toprule
$n$      &$|u-u_h|_{2,h}$ & order & $|u-u_h|_{1,h}$ & order &$\|u-u_h\|_0$ & order  \\
\midrule
$4$      &5.780E1     &       &4.523E0    &               &9.524E-1     &       \\
$8$      &3.254E1     &0.83   &1.664E0    &1.44           &3.346E-1     &1.51   \\
$16$     &1.693E1     &0.94   &4.692E-1   &1.83           &9.242E-2     &1.86   \\
$32$     &8.568E0     &0.98   &1.216E-1   &1.95           &2.377E-2     &1.96   \\
$64$     &4.297E0     &1.00   &3.070E-2   &1.99           &5.990E-3     &1.99   \\
\bottomrule
\end{tabular}
\caption{The discrete $H^2$, $H^1$ and $L^2$ errors produced by $W_h$ applied to the fourth order elliptic problem determined by (\ref{e: example 2}) for various $n$.
 \label{t: example 2}}
\end{center}
\end{table}

\end{document}